\documentclass[12pt]{article}

\usepackage{amsmath, amsthm, amssymb, amscd}
\usepackage{bm, mathrsfs, dsfont}
\usepackage{cases}
\usepackage{bbding}
\usepackage{cite, hyperref}
\usepackage[all,pdf]{xy}
\usepackage{upgreek}
\usepackage{fancyhdr}
\def\nd{\noindent}
\newtheorem{thm}{Theorem}[section]
\newtheorem{lem}[thm]{Lemma}
\newtheorem{prop}[thm]{Proposition}
\newtheorem{cor}[thm]{Corollary}

\newtheorem{rem}[thm]{Remark}

\title{\bfseries A note on the Hurwitz problem\\ and cone spherical metrics}
\author{Jijian Song$^1$, Bin Xu$^2$ and Yu Ye$^3$}

\begin{document}

\maketitle

\nd{\small  $^{1}$Center for Applied Mathematics, Tianjin University. No. 135 Yaguan Road, Tianjin 300350 China. \\
$^{2,3}$CAS Wu Wen-Tsun Key Laboratory of Mathematics and School of Mathematical Sciences, University of Science and Technology of China, No. 96 Jinzhai Road, Hefei, Anhui Province 230026 China.}

\nd {\small $^{1}$jijian.song@kcl.ac.uk \quad\quad  $^2$bxu@ustc.edu.cn \quad\quad $^{3}$yeyu@ustc.edu.cn }
\par\vskip0.5cm

\nd {\small {\bf Abstract:}  Motivated by cone spherical metrics on compact Riemann surfaces of positive genus, we address a specific case of the Hurwitz problem. Specifically, for given positive integers $d$, $g$ and $\ell$, along with a collection $\Lambda$ consisting of $(\ell+2)$ partitions of the positive integer $d$:
\[
(a_1,\cdots, a_p),  (b_1,\cdots, b_q),  (m_1+1,1,\cdots, 1),  \cdots,  (m_{\ell}+1,1,\cdots, 1),
\]
where $(m_1,\cdots, m_{\ell})$ is a partition of $p+q-2+2g$, we establish the existence of a branched cover from a compact Riemann surface of genus $g$ to the Riemann sphere ${\Bbb P}^1$ with branch data $\Lambda$. Notably, this complements the genus-zero case, as initially discovered by the first two authors in {\it Algebra Colloq.} {\bf 27} (2020), no. 2, 231-246. Their work was prompted by analogous metrics on ${\Bbb P}^1$, leading them to conjecture the validity of the aforementioned statement in that context.}

\nd{\bf Keywords.} branched cover, Hurwitz problem, cone spherical metric, one-form

\nd {\bf 2020 Mathematics Subject Classification.} Primary 57M12; secondary 20B35

\section{Introduction}

A branched cover $f$ of degree $d>1$ mapping a compact Riemann surface to another induces a collection $\Lambda$ comprising finitely many partitions of $d$, each with a length less than $d$, referred to as the {\it branch data} of $f$. Notably, $\Lambda$ adheres to the well-known Riemann-Hurwitz formula. Simultaneously, there exists a set of partitions of $d$, termed {\it exceptions}, that satisfy the Riemann-Hurwitz formula but do not qualify as the branch data for any branched cover \cite{Zh06}.

The classical Hurwitz problem, spanning over 130 years of history, seeks to {\it enumerate branched covers with prescribed branch data under some equivalence relation}. A simplified variant involves determining {\it whether a branched cover exists with a given branch data}, a question resolved in \cite[Proposition 3.3]{EKS84}, asserting the absence of exceptions when the target surface has positive genus. Our focus in this manuscript centers on branched covers to the Riemann sphere ${\Bbb P}^1$. Despite substantial progress by mathematicians in understanding both branch data and exceptions, both the original problem and its simplified version remain open. Comprehensive works, both classical \cite{Bo82, EKS84, Eze78, Fr76, Ger87, Huse62, KZ95, Sin70, Thom65} and recent \cite{Ba01, LZ04, MSS04, OP06, Pa09, PaPe09, PaPe12, PerPe06, PerPe08, Zh06, Zhu19, SX20}, have contributed significantly to this complex problem. The new results presented in this manuscript pertain exclusively to the simplified version of the Hurwitz problem, employing notations from \cite{CWWX15, SX20}.
%Nearly 40 years ago, Edmonds-Kulkarni-Stong \cite{EKS84} proposed the so called {\it prime-degree conjecture} which says that there be no exception if $d$ is a prime number. It seems that there hasn't been much progress on this conjecture since then except that  Pascali-Petronio \cite{PaPe12} provided some supporting evidence.

{\it Cone spherical metrics} are conformal metrics on compact Riemann surfaces with constant curvature $+1$ and finitely many cone singularities. The problem of establishing the existence of such metrics with prescribed cone singularities on compact Riemann surfaces has remained open since its formal proposition by Troyanov \cite{Troyanov91}, with the sub-critical case being resolved in the 1980s.

A metric is termed {\it reducible} if its developing map has monodromy in ${\rm U}(1)$; otherwise, it is categorized as {\it irreducible}. Q. Chen, W. Wang, Y. Wu and the second author \cite{CWWX15} provided a characterization of reducible (cone spherical) metrics in terms of meromorphic one-forms with simple poles and periods in $\sqrt{-1}{\Bbb R}$, denoted as {\it unitary one-forms} on compact Riemann surfaces. Notably, the cone angles of a reducible metric are intricately determined by the residues of poles and the multiplicities of zeros of a unitary one-form (\cite[Theorem 1.5]{CWWX15}).

A unitary one-form on ${\Bbb P}^1$ has residues in ${\Bbb R}\setminus \{ 0 \}$ and adheres to both the residue theorem and the degree condition---wherein the sum of multiplicities of its zeros equals the number of its simple poles minus 2. In pursuit of establishing the angle constraint for reducible metrics with cone angles lying in $2\pi{\Bbb Q}_{>0}$ on ${\Bbb P}^1$, the second author, building upon the draft of \cite{CWWX15}, discovered in 2014 that it suffices to prove the following fact: \\

\nd {\bf Fact 1.} {\it Consider two partitions $(a_1,\cdots, a_p)$ and $(b_1,\cdots, b_q)$ of a positive integer $d>1$, along with a partition $(m_1,\cdots, m_\ell)$ of $p+q-2>0$. The existence of a unitary one-form $\omega$ on ${\Bbb P}^1$ having $p+q$ residues of $a_1,\cdots, a_p, -b_1,\cdots,-b_q$, and $\ell$ zeros with multiplicities $m_1,\cdots, m_\ell$ is guaranteed if and only if}
\begin{equation}
\label{equ:degwt}
\max(m_1,\cdots, m_\ell)<\frac{d}{{\rm GCD}(a_1,\cdots,a_p, b_1,\cdots, b_q)}.
\end{equation}

\nd Upon obtaining such a one-form $\omega$, the second author noted in 2014 that the solution to the ordinary differential equation $$\frac{{\rm d}f}{f}=\omega$$ is a unique rational function $f$ on ${\Bbb P}^1$, up to a multiple. This observation substantiates the necessary condition \eqref{equ:degwt} from Fact 1. The second author further reduced the sufficient part of Fact 1 to a specific case of the Hurwitz problem, which was subsequently resolved affirmatively in 2015 in collaboration with the first author. Refer to \cite[Theorem 1.1]{SX20}, specifically Case 1 of Theorem \ref{thm:rat}. By taking the logarithmic differential $\frac{{\rm d}f}{f}$ with respect to the branched cover $f \colon {\Bbb P}^1\to {\Bbb P}^1$ in Case 1 of Theorem \ref{thm:rat}, the desired one-form $\omega$ is derived.

Eremenko \cite{Eremenko:2017} utilized this specific case of the Hurwitz problem, solved by the first and second authors, and incorporated the theory of o-minimal structures to establish the angle constraint for reducible metrics on ${\Bbb P}^1$. Preceding Eremenko, Mondello-Panov \cite[Theorem C]{MP1505} utilized parabolic rank two stable bundles to provide the angle constraint for irreducible metrics on ${\Bbb P}^1$. The amalgamation of these results yields the angle constraint for cone spherical metrics on ${\Bbb P}^1$. X. Zhu \cite{Zhu19} subsequently utilized this angle constraint to discover infinitely many new exceptions for branched covers from ${\Bbb P}^1$ to itself.

Mondello-Panov \cite[Theorem A]{MP1807} employed the cutting and gluing technique to demonstrate that the Gauss-Bonnet formula constitutes the sole angle constraint for cone spherical metrics on compact Riemann surfaces of {\it positive} genus. More recently, Q. Gendron and G. Tahar \cite[Theorem 5.3]{GT21} applied totally real Jenkins-Strebel differentials to establish the angle constraint for reducible metrics on compact Riemann surfaces of {\it positive} genus. This was reduced to an existence theorem of unitary one-forms with both prescribed residues and zero multiplicities. Almost simultaneously, Z. Wei, Y. Wu and the second author \cite[Theorem 1.9]{WWX22} employed the football-decomposition technique for reducible metrics to prove the same result as the theorem of Gendron-Tahar. A special case of this result is stated as follows: \\

\nd {\bf Fact 2.} {\it Let $p,q,g$ be three positive integers, $(a_1,\cdots, a_p)$ and $(b_1,\cdots, b_q)$ be two partitions of  $d$, and $(m_1,\cdots, m_\ell)$ a partition of $p+q-2+2g$. Then there exists a unitary one-form $\omega$ on some compact Riemann surface of genus $g$ such that it has the $p+q$ residues of $a_1,\cdots, a_p, -b_1,\cdots,-b_q$ and $\ell$ zeros with multiplicities $m_1,\cdots, m_\ell$, respectively. }\\

\nd In essence, the positive genus introduced in Fact 2 serves to eliminate the algebraic constraint \eqref{equ:degwt} in Fact 1, allowing the residue theorem, coupled with the degree condition, to ensure the existence of $\omega$. However, it is crucial to note that Fact 2 is non-trivial, given that all periods of $\omega$ reside in $\sqrt{-1}{\Bbb R}$. Inspired by Fact 2, we establish Case 2 in the subsequent theorem, the proof of which stands independent of this particular observation.
\begin{thm}
\label{thm:rat}
Let $d$ and $\ell$ be positive integers. Suppose that a collection $\Lambda$ comprises $\ell + 2$ partitions of $d$ as follows:
\[ (a_{1}, a_{2}, \ldots, a_{p}), \; (b_{1}, b_{2}, \ldots, b_{q}),\; (m_{1}+1, 1, \ldots, 1),\; \ldots, \; (m_{\ell}+1, 1, \ldots, 1).\]
We define the {\rm total branch number} $v(\Lambda)$ of $\Lambda$ to be $m_1+\cdots+m_\ell+(d-p)+(d-q)$.
Then $\Lambda$ represents the branch data of a branched cover from a compact Riemann surface $X$ to ${\Bbb P}^1$ if and only if it satisfies one of the following two conditions{\rm :}
\begin{enumerate}
\item $v(\Lambda) = 2d - 2$ and $\displaystyle{\max\{m_{1}, \ldots, m_{\ell}\} < \frac{d}{\operatorname{GCD}(a_{1}, \ldots, a_{p}, b_{1}, \ldots, b_{q})}{\rm ;}}$
\item The total branching order $v(\Lambda) \geq 2d$ is even.
\end{enumerate}
Moreover, the genus of $X$ equals $\displaystyle{\frac{v(\Lambda)-2d+2}{2}}$.
\end{thm}

A branched cover from a compact Riemann surface $X$ to ${\Bbb P}^1$ is termed a {\it Belyi function} on $X$ if it has at most three branched points. Employing a similar argument as in the proof of Theorem 1.3 in \cite{SX20}, we can deduce the following corollary from Theorem \ref{thm:rat}.

\begin{cor} Utilizing the notations introduced in Theorem \ref{thm:rat}, there exists a Belyi function of degree $\ell d$ on $X$ with the branch data
\[\widetilde\Lambda=\left\{(\ell a_1,\cdots, \ell a_p),(\ell b_1,\cdots, \ell b_q), (m_1+1,\cdots, m_\ell+1,1,\cdots, 1)\right\}.\]
\end{cor}

\begin{rem}
{\rm The case of $\ell=1$ in Theorem \ref{thm:rat} has been established by Boccara \cite{Bo82}, providing additional motivation alongside Facts 1-2. While we have previously proven the first instance with $v(\Lambda) = 2d - 2$ in \cite{SX20}, specifically when $X$ is ${\Bbb P}^1$. The necessary part of the second case follows directly from Riemann-Hurwitz. Consequently, in the remainder of this note, we focus on demonstrating its sufficient part.

An evident distinction emerges between Case 1 and Case 2 in Theorem \ref{thm:rat}. This disparity can be likened to irregularity phenomena observed in the enumeration geometry of Riemann surfaces: unstable Riemann surfaces exhibit more irregularity than stable surfaces \cite[p.53]{Eyn16}.

By taking the logarithmic differential of the branched cover, Theorem \ref{thm:rat} implies Fact 1, with the genus of $X$ being zero. However, Fact 2 does not directly result from the theorem. In fact, Wei, Wu and the second author established Fact 2 in \cite[Theorem 1.9]{WWX22} by constructing the very reducible metric corresponding to the unitary one-form involved. Conversely, following Hurwitz's approach, we apply the Riemann existence theorem \cite[Theorem 2, p.49]{Don11} to reduce Theorem \ref{thm:rat} to the following proposition.}
\end{rem}

%Note that if $v(\Lambda) \geq 2d$ is even, then the realizability of $\Lambda$ is equivalent to the existence of a rational function with the branch data $\Lambda$ on some compact Riemann surface $X$ of genus $g \geq 1$. Moreover, we can derive $m_{1} + \ldots + m_{\ell} = p + q - 2 + 2g$ from the Riemann-Hurwitz formula.

\begin{prop}
\label{prop:subgroup}
Let $d$, $\ell$ and $g$ be positive integers. If
\[ \Lambda = \big\{(a_{1}, a_{2}, \ldots, a_{p}), \; (b_{1}, b_{2}, \ldots, b_{q}),\; (m_{1}+1, 1, \ldots, 1),\; \ldots, \; (m_{\ell}+1, 1, \ldots, 1)\big\} \]
represents a collection of $\ell + 2$ partitions of $d$ with $m_{1} + \cdots +m_{\ell} = p+q-2+2g$, then there exist $\ell +2$ permutations $\tau_{1}, \tau_{2}, \sigma_{1}, \ldots, \sigma_{\ell} \in S_{d}$ satisfying the following three conditions{\rm :}
\begin{enumerate}
\item $\tau_{1} \tau_{2} \sigma_{1} \cdots \sigma_{\ell} = id${\rm ;}
\item $\tau_{1}$ has type $a_{1}^{1} a_{2}^{1} \cdots a_{p}^{1}$, $\tau_{2}$ has type $b_{1}^{1}b_{2}^{1} \cdots b_{q}^{1}$, and $\sigma_{i}$ has type $(m_{i}+1)^{1} 1^{d-m_{i}-1}$ for $i = 1,2, \ldots, \ell${\rm ;}
\item the subgroup generated by $\tau_{1}, \tau_{2}, \sigma_{1}, \ldots, \sigma_{\ell}$ acts transitively on the set $\{1,2,\cdots, d\}$.
\end{enumerate}
\end{prop}

We establish the aforementioned proposition in the subsequent section. In the final section, we pose a question regarding the enumeration of branched covers in Theorem \ref{thm:rat} under weak and strong equivalence relations, respectively, which salutes to Adolf Hurwitz.
%Readers can see easily the organization of this manuscript from the titles of the left two sections.

\section{Proof of Propostion \ref{prop:subgroup}}

\subsection{Case $\ell=1$: three partitions}
This particular case was, in fact, proven by Boccara \cite{Bo82}. In this subsection, we present an alternative proof, which appears to be more constructive than the original one proposed by Boccara.
\begin{prop}
\label{prop:lis1}
Proposition \ref{prop:subgroup} holds true for $\ell = 1$.
\end{prop}
\begin{proof}
We use induction on $d$.

Firstly, observe that $d \geq m_{1} +1 = p + q - 2 + 2g +1 \geq 3$. Additionally, if $d = 3$, then $\Lambda = \{(3),\; (3), \; (3)\}$. In this case, we can choose $\tau_{1} = \tau_{2} = \sigma = (123)$.

Now, let's assume that the proposition holds true for partitions of $d \leq D - 1$, where $D\geq 4$. Recall that $\Lambda = \{ (a_{1}, \ldots, a_{p}),\; (b_{1}, \ldots, b_{q}),\; (m_{1}+1, 1, \ldots, 1) \}$ is a collection of three partitions of $D$ and $2g = m_{1} - p - q + 2\geq 2$.
\begin{itemize}
\item[Case 1] Suppose that there exist $i$ and $j$ such that $a_{i} \neq b_{j}$. Without loss of generality, we assume that $a_{p} > b_{q}$ and $b_{q} = \min \{a_{1}, \ldots, a_{p}, b_{1}, \ldots, b_{q} \}$.
\begin{itemize}
\item[Subcase 1.1] Let $m_{1} - 1 < D - b_{q}$. Then, applying the induction hypothesis to the collection
\[ \widetilde{\Lambda} = \{ (a_{1}, \ldots, a_{p}-b_{q}),\; (b_{1}, \ldots, b_{q-1}),\; (m_{1}, 1,\ldots, 1) \},\]
of three partitions of $D - b_{q}$, we could find three permutations $\widetilde{\tau}_{1},\; \widetilde{\tau}_{2}, \; \widetilde{\sigma}_{1} \in S_{D-b_{q}}$ such that $\widetilde{\tau}_{2} \widetilde{\sigma}_{1} = \widetilde{\tau}_{1} = \tilde{\mu}_{1} \cdots \tilde{\mu}_{p}$, where $\tilde{\mu}_{i}$ is a cycle of length $a_{i}$ for $1 \leq i \leq p-1$ and $\tilde{\mu}_{p}$ is a cycle of length $(a_{p}-b_{q})$. Since
$\langle \widetilde\tau_1,\widetilde\tau_2, \widetilde{\sigma}_{1} \rangle$ acts transitively on $\{1,2,\cdots, D-b_q\}$, by \cite[Lemma 2.10]{SX20}, $\widetilde{\sigma}_{1}$ and $\tilde{\mu}_{p}$ have at least one common element, say $a$. Then we could take
\begin{align*}
\tau_{2} &= (D-b_{q}+1, D-b_{q}+2, \ldots, D) \widetilde{\tau}_{2}, \\
\sigma_{1} &= \widetilde{\sigma}_{1}(a \, D), \\
\tau_{1} &= \tilde{\mu}_{1} \cdots \tilde{\mu}_{p-1}\cdot \Big( (D-b_{q}+1, D-b_{q}+2, \ldots, D) \tilde{\mu}_{p} (a\, D) \Big).
\end{align*}
\item[Subcase 1.2] Let $m_{1} - 1 \geq D - b_{q}$. Since $m_{1}+1 \leq D$, we have $b_{q} \geq 2$. Moreover, if $b_{q} = 2$, then $m_{1}+1 = D$. By Edmonds-Kulkarni-Stong \cite[Proposition 5.2]{EKS84}, it holds true for Case $b_{q} = 2$. Hence, we assume $b_{q} \geq 3$ and consider the following new collection
\[ \widetilde{\Lambda} = \{ (a_{1}, \ldots, a_{p-1}, a_{p}-2),\; (b_{1}, \ldots, b_{q-1}, b_{q}-2),\; (m_{1}-1,1, \ldots, 1) \}\]
of three partitions of $D-2$. Since $D - b_{q} \leq m_{1} - 1 = p + q - 3 + 2g$,
we have $2g \geq 3 + D - b_{q} - p - q$.\\

At first, it is easy to check that if $\max\{p, q\} \leq 2$, then $2g \geq 3$. On the other hand, $D - b_{q} - p - q \geq \frac{D}{3} - b_{q}$ since $p, q \leq \frac{D}{3}$. Therefore, if $\max\{p, q\} \geq 3$,  then we have
$$2g \geq 3 + D - b_{q} - p - q \geq 3 + \frac{\max\{p , q\} \cdot b_{q}}{3} - b_{q} \geq 3.$$
Summing up, we always have $g \geq 2$ and $m_{1} - 2 = p + q - 2 + 2(g-1)$.\\

By the induction hypothesis, there exist $\widetilde{\tau}_{1},\, \widetilde{\tau}_{2},\, \widetilde{\sigma}_{1}$ in $S_{D-2}$ corresponding to $\widetilde{\Lambda}$ with $\widetilde{\sigma}_{1} = \widetilde{\tau}_{1} \widetilde{\tau}_{2}$. Let $\widetilde{\tau}_{1} = \tilde{\mu}_{1} \cdots \tilde{\mu}_{p}$, where $\tilde{\mu}_{i}$ is a cycle of length $a_{i}$ for $1 \leq i \leq p-1$ and $\tilde{\mu}_{p}$ is a cycle of length $(a_{p}-2)$. Let $\widetilde{\tau}_{2} = \tilde{\nu}_{1} \cdots \tilde{\nu}_{q}$, where $\tilde{\nu}_{i}$ is a cycle of length $b_{i}$ for $1 \leq i \leq q-1$ and $\tilde{\nu}_{q}$ is a cycle of length $(b_{q}-2)$. By the transitive property, $\widetilde{\sigma}_{1}$ and $\tilde{\mu}_{p}$ have at least one common element, say $x$; $\widetilde{\sigma}_{1}$ and $\tilde{\nu}_{q}$ have at least one common element, say $y$. At last, we take
\begin{align*}
\tau_{1} &= (D,\, D-1,\, x) \widetilde{\tau}_{1} \\
\tau_{2} &= \widetilde{\tau}_{2} (D,\, D-1,\, y) \\
\sigma_{1} &= \tau_{1} \tau_{2},
\end{align*}
where $\sigma_{1} = (D,\, D-1,\, x) \widetilde{\sigma}_{1} (D,\, D-1,\, y)$ is always an $(m_{1} + 1)$-cycle in $S_{D}$ whether $x$ coincides with $y$ or not.
\end{itemize}
\item[Case 2] Let $a_{1} = \cdots = a_{p} = b_{1} = \cdots = b_{q} = k$. Then we have $k \geq 3$ since
$$p \cdot k = D \geq m_{1} + 1 = 2p - 1 + 2g.$$
Moreover, by Proposition 5.2 in \cite{EKS84}, we could assume $p = q \geq 2$ and then $m_{1} \geq 4$.
\begin{itemize}
\item[Subcase 2.1] Let $m_{1} - 1 \leq D - k$. Then, considering the following new collection
\[ \widetilde{\Lambda} = \{(a_{1}, \ldots, a_{p-1}),\; (b_{1}, \ldots, b_{q-1}),\; (m_{1}-1, 1, \ldots, 1)\} \]
of three partitions of $D-k$, we find by the induction hypothesis $\widetilde{\tau}_{1},\, \widetilde{\tau}_{2},\, \widetilde{\sigma}_{1} \in S_{D-k}$ corresponding to $\widetilde{\Lambda}$ such that $\widetilde{\sigma}_{1} = \widetilde{\tau}_{1} \widetilde{\tau}_{2}$. Let $\widetilde{\tau}_{1} = \tilde{\mu}_{1} \cdots \tilde{\mu}_{p}$, where $\tilde{\mu}_{i}$ is a cycle of length $k$ for $1 \leq i \leq p$. We claim that {\it each $\tilde{\mu}_{i}$ has at least $2$ common elements with $\widetilde{\sigma}_{1}$}. In fact, if $\tilde{\mu}_{i}$ has exactly one common elements with $\widetilde{\sigma}_{1}$, say $a$, then the cycle in $\widetilde{\tau}_{2} = \widetilde{\tau}^{-1}_{1} \widetilde{\sigma}_{1}$ containing $a$ has length greater than $k$. There is a contradiction here.

Let $x_{1}, x_{i}$ be two common numbers lying in both $\tilde{\mu}_{1} = (x_{1},x_{2}, \ldots, x_{i}, \ldots, x_{k})$ and $\widetilde{\sigma}_{1}$. Then we could take
\begin{align*}
\tau_{1} &= (x_{1}, D-k+1)(x_{i}, D-k+i)(D-k+1, D-k+2, \ldots, D) \widetilde{\tau}_{1}, \\
\tau_{2} &= \widetilde{\tau}_{2} (D, D-1, \ldots, D-k+1), \\
\sigma_{1} &= \tau_{1} \tau_{2},
\end{align*}
where $\sigma_{1} = (x_{1}, D-k+1) (x_{i}, D-k+i) \widetilde{\sigma}_{1}$ is a cycle of length $(m_{1} +1)$.
\item[Subcase 2.2] Let $m_{1} - 1 > D - k$. Then we consider the following collection
\[ \widetilde{\Lambda} = \{ (a_{1}, \ldots, a_{p-1}, a_{p}-2),\; (b_{1}, \ldots, b_{q-1}, b_{q}-2),\; (m_{1}-1,1, \ldots, 1) \} \]
of three partitions of $D-2$. Since $g \geq 2$, we could take the same construction as in Subcase 1.2 except for Case $\Lambda = \{(3,3), (3,3), (5,1)\}$ for which we could construct the three permutations by hand.
\end{itemize}
\end{itemize}
\end{proof}

\subsection{Case $\ell>1$: more than three partitions}
In this subsection, we will establish Case $\ell>1$ of Proposition \ref{prop:subgroup} through induction on both $\ell$ and $m_{1} + \cdots + m_{\ell} - p - q$. To begin, we require some preliminary lemmas.

\begin{lem}
\label{lem:3cycles}
Let $1 \leq s,\; r \leq d$ be two integers such that
$$s + r \geq d+1\quad {\rm and}\quad s + r \equiv d + 1 ({\rm mod}\ 2).$$
Then there exist an $s$-cycle $\sigma_{1}$ and an $r$-cycle $\sigma_{2}$ such that $\sigma_{1} \sigma_{2}$ is a $d$-cycle.
\begin{proof}
If $2k = s + r - (d+1)$, then $0 \leq 2k < r$. Let
\begin{align*}
\sigma_{2}^{-1} &= (r, r-1, \ldots, 2,1)(2k+1,2k, \ldots, 2,1) \\
&=(r, r-1, \ldots, 2k+2, 2k+1, 2k-1, \ldots, 1, 2k, 2k-2, \ldots, 2)
\end{align*}
Hence, $\sigma_{2}$ is an $r$-cycle. Note that
\[ (1,2, \ldots, d) \sigma_{2}^{-1} = (1, 2k+1,2k, \ldots, 2,r+1, r+2, \ldots, d) \]
Set $\sigma_{1} = (1, 2k+1,2k, \ldots, 2,r+1, r+2, \ldots, d)$. Then $\sigma_{1}$ is an $s$-cycle and $\sigma_{1} \sigma_{2} = (1,2, \ldots, d)$ is a $d$-cycle.
\end{proof}
\end{lem}

%\begin{lem}
%\label{lem:pqare1}
%If $p = q = 1$, then Proposition \ref{prop:subgroup} holds true for any $\ell \geq 1$.
%\end{lem}
%\begin{proof}
%Note that if $p = q = 1$, then $a_{1} = b_{1} = d$ and $m_{1} + \cdots + m_{\ell} = 2g$ is an even number. It is easy to choose $(m_{i}+1)$-cycles $\sigma_{i} \in S_{d}$ such that $\sigma_{1} \sigma_{2} \cdots \sigma_{\ell}$ is a cycle of odd length. By the construction in the proof of Lemma \ref{lem:3cycles} and renumbering numbers in the cycle $(\sigma_{1} \cdots \sigma_{\ell})$ if necessary, there exist two $d$-cycles $\tau_{1}$ and $\tau_{2}$ such that
%\[ \tau_{2} \cdot (\sigma_{1} \sigma_{2} \cdots \sigma_{\ell}) = \tau_{1}^{-1}. \]
%\end{proof}

\begin{lem}
\label{lem:baselis2}
Proposition \ref{prop:subgroup} holds true if $\ell = 2$ and $p = q = m_{1} = m_{2}$.
\end{lem}
\begin{proof} Case $p=q=m_1=m_2=1$ follows from \cite[Proposition 5.2]{EKS84}. We assume $p=q=m_1=m_2>1$ in what follows.
We could derive $p, q < d$ from $(m_{1} +1) \leq d$. Without loss of generality, we could assume that
\[ a_{1} \geq a_{2} \geq \cdots \geq a_{p}, \quad b_{1} \geq b_{2} \geq \cdots \geq b_{q}. \]
Then $a_{1} \geq 2$ and $b_{1} \geq 2$. Let
\begin{align*}
\tau_{1} &= (a_{1}+a_{2}+\cdots+a_{p-1}+1, a_{1}+a_{2}+\cdots+a_{p-2}+1, \ldots, a_{1}+1, 1) \, (1, 2, \ldots, d) \\
&=(1,2,\ldots, a_{1}) \, (a_{1}+1, a_{1}+2, \ldots, a_{1}+a_{2}) \cdots (a_{1}+\cdots+a_{p-1}+1, \ldots, d) \\
\tau_{2} &= (d, d-1,\ldots, 1)(1, b_{1}+1, \ldots, b_{1}+b_{2}+\cdots+b_{q-1}+1) \\
&=(d, d-1, \ldots, b_{1}+b_{2}+\cdots+b_{q-1}+1)\cdots (b_{1}+b_{2}, \ldots, b_{1}+2, b_{1}+1) \, (b_{1}, \ldots, 2, 1)
\end{align*}
Then we have the following equality relevant to the four permutations of $\tau_1,\,\tau_2,\, \sigma_1$ and $\sigma_2$:
\begin{align*}
\tau_{1}\tau_{2} &=(a_{1}+a_{2}+\cdots+a_{p-1}+1,\ldots, a_{1}+1, 1) (1, b_{1}+1,\ldots, b_{1}+b_{2}+\cdots+b_{q-1}+1) \\
&=\underbrace{(a_{1}+a_{2}+\cdots+a_{p-1}+1,\ldots, a_{1}+1, 1)(12)}_{\sigma_{2}^{-1}} \underbrace{(12)(1, b_{1}+1,\ldots, b_{1}+b_{2}+\cdots+b_{q-1}+1)}_{\sigma_{1}^{-1}}
\end{align*}
Hence we complete the proof.
\end{proof}

\begin{lem}
\label{lem:caselis2eq}
Proposition \ref{prop:subgroup} holds true if $\ell = 2$ and $p+q = m_{1}+m_{2}$.
\end{lem}
\begin{proof}
Without loss of generality, we assume that $p \geq q$ and $m_{1} \geq m_{2}$. Moreover, by Lemma \ref{lem:baselis2}, we could assume $p \neq q$ or $m_{1} \neq m_{2}$, that is $m_{2} < p$. Consider the following new collection of partitions of $d$
\[ \widetilde{\Lambda} = \big\{ (a_{1}+a_{2}+\cdots+a_{m_{2}+1},a_{m_{2}+2}, \ldots, a_{p}),\; (b_{1}, b_{2}, \ldots, b_{q}), \; (m_{1}+1,1,\ldots, 1)\big\}. \]
Then by Proposition \ref{prop:lis1}, we know that there exist $\widetilde{\tau}_{1},\; \tau_{2}, \; \sigma_{1}$ corresponding to $\widetilde{\Lambda}$ such that $\widetilde{\tau}_{1} \tau_{2} = \sigma_{1}^{-1}$. In particular, $\widetilde{\tau}_{1}$ is a product of mutually disjoint $(p-m_{2})$ cycles of lengths $(a_{1}+a_{2}+\cdots+a_{m_{2}+1}),\; a_{m_{2}+2}, \ldots, a_{p}$, respectively. For simplicity, we assume that the cycle of length $(a_{1}+a_{2}+\cdots+a_{m_{2}+1})$ in the product is
\[ (1,2, \ldots, a_{1}+\cdots+a_{m_{2}+1}). \]
Note that
\begin{align*}
&(a_{1}+\cdots+a_{m_{2}}+1,\; a_{1}+\cdots+a_{m_{2}-1}+1,\; \ldots, a_{1}+1,\; 1) \cdot (1,2,\ldots, a_{1}+a_{2}+\cdots+a_{m_{2}+1}) \\
=&(1,2,\ldots, a_{1})(a_{1}+1,a_{1}+2,\ldots, a_{1}+a_{2}) \cdots (a_{1}+\cdots+a_{m_{2}}+1,\ldots, a_{1}+\cdots+a_{m_{2}+1})
\end{align*}
Set
\begin{align*}
\sigma_{2}^{-1} &= (a_{1}+\cdots+a_{m_{2}}+1,\; a_{1}+\cdots+a_{m_{2}-1}+1,\; \ldots, a_{1}+1,\; 1), \\
\tau_{1} &= \sigma_{2}^{-1} \widetilde{\tau}_{1}.
\end{align*}
Then we have
$\tau_{1} \tau_{2} \sigma_{1} \sigma_{2} = id$.
\end{proof}

\begin{lem}
\label{lem:lis2}
Proposition \ref{prop:subgroup} holds true for $\ell = 2$ and $p+q \leq m_{1}+m_{2}$.
\end{lem}
\begin{proof}
We use induction on $m_{1}+m_{2} - (p+q)$. By Lemma \ref{lem:caselis2eq}, we assume it holds true for $m_{1}+m_{2} = (p+q) + 2k$. Suppose $m_{1}+m_{2} - (p+q) = 2k + 2$. Then $(m_{1}-1) + (m_{2}-1) = (p+q) + 2k$. By the induction hypothesis, there exist $\tau_{1}, \tau_{2}$ and $m_{1}$-cycle $\widetilde{\sigma}_{1}$, $m_{2}$-cycle $\widetilde{\sigma}_{2}$ such that $\tau_{1} \tau_{2} = \widetilde{\sigma}_{2} \widetilde{\sigma}_{1}$. Let $A_{i}$ be the set of numbers in the cycle $\widetilde{\sigma}_{i} (i=1,2)$. We choose two numbers $x, y \in \{1,2,\ldots, d\}$ according to the following three cases, respectively:
\begin{itemize}
\item If $A_{1} \subseteq A_{2}$, then pick $x \in A_{1},\; y \notin A_{2}$;
\item If $A_{2} \subseteq A_{1}$, then pick $x \in A_{2},\; y \notin A_{1}$;
\item If $A_{1} \nsubseteq A_{2}$ and $A_{2} \nsubseteq A_{1}$, then pick $x \in A_{1} \setminus A_{2},\; y \in A_{2} \setminus A_{1}$.
\end{itemize}
Then $\sigma_{2} = \widetilde{\sigma}_{2}(x y)$ is a $(1+m_{2})$-cycle, $\sigma_{1} = (x y)\widetilde{\sigma}_{1}$ is a $(1+m_{1})$-cycle and
\[ \tau_{1} \tau_{2} = \sigma_{2} \sigma_{1}. \]
\end{proof}

Now we are in a position to complete the proof of  Proposition \ref{prop:subgroup}.
\begin{proof}[Proof of Proposition \ref{prop:subgroup}]
Here we use induction on $\ell$ and $m_{1}+\cdots+m_{\ell} - (p+q)$. Assume that Proposition \ref{prop:subgroup} holds true for $\ell \leq k-1\; (k \geq 3)$ and consider the case $\ell = k$. Then we use induction on $m_{1}+\cdots+m_{\ell} - (p+q)$.
\begin{itemize}
\item[(Step 1)] Let $m_{1} + \cdots + m_{\ell} = p+q$. Then  we could assume $m_{1} \geq m_{2} \geq \cdots \geq m_{\ell}$ and $p \geq q$ without loss of generality. Hence $p > m_{\ell}$ and the same argument as in Lemma \ref{lem:caselis2eq} works.
\item[(Step 2)] Let $m_{1}+\cdots+m_{\ell} = p+q + 2n-2$ where $n>1$. Then the same argument as in Lemma \ref{lem:lis2} shows that it holds true for $m_{1}+\cdots+m_{\ell} = p+q+2n$.
\end{itemize}
\end{proof}

\section{Hurwitz numbers}
Two branched covers $f_1 \colon X_1\to Y$ and $f_2 \colon X_2\to Y$ between compact Riemann surfaces are said to be {\it weakly equivalent} if there exist two biholomorphic maps $\widetilde g \colon X_1\to X_2$ and $g \colon Y \to Y$ such that $f_2\circ \widetilde g = g\circ f_1$. They are deemed {\it strongly equivalent} if the set of branched points in $Y$ remains fixed, and one can choose $g = id_{Y}$. The {\it weak} (resp. {\it strong}) {\it Hurwitz number} of a branch datum is the count of weak (resp. strong) equivalence classes of branched covers realizing it. The {\it simple Hurwitz numbers} specifically refer to those strong Hurwitz numbers associated with branch data in the form
\[(2,1,\cdots,1),\,\cdots,\, (2,1,\cdots, 1),\,(a_1,\cdots, a_p).\]

Long ago, Mednykh in \cite{Me84, Me90} provided general formulae for computing strong Hurwitz numbers, though their actual implementation is generally intricate. Lando-Zvonkin \cite[Chapter 5]{LZ04} systematically expounded on how to compute simple Hurwitz numbers for the target ${\Bbb P}^1$ in terms of intersection numbers on moduli spaces of curves. Furthermore, they presented an enumeration of polynomial rational functions ${\Bbb P}^1\to {\Bbb P}^1$ in terms of the Lyashko-Looijenga mapping. Monni-Song-Song \cite{MSS04} employed algebraic methods to compute simple Hurwitz numbers concerning branch data of the form $\{(2,1,\cdots,1),\,\cdots,\, (2,1,\cdots, 1)\}$ for arbitrary source and target Riemann surfaces, deriving the generating function for such simple Hurwitz numbers. Dubrovin-Yang-Zagier \cite{DYZ17} presented a polynomial-time algorithm for computing these simple Hurwitz numbers for the target surface ${\Bbb P}^1$. Quite recently, Petronio \cite{Pet19, Pet20} and Petronio-Sarti \cite{PetSa19} explicitly computed the weak Hurwitz number for specific branch data consisting of three partitions and having the target surface ${\Bbb P}^1$, utilizing a combinatorial method based on Grothendieck's dessin d'enfants. \\

\nd {\bf Question} Derive a concise formula for all weak (or strong) Hurwitz numbers corresponding to the branch data outlined in Theorem \ref{thm:rat}. \\

%Is it possible to connect the Hurwitz numbers with respect to the branched covers in Theorem \ref{thm:rat} with
%some Hodge integrals in the moduli space $\mathcal{M}_{0,r+2}$ of the marked Riemann sphere?

%\section{Proof of Corollary 1.2}

\nd {\bf Acknowledgments.}
J.S. expresses gratitude for the hospitality extended during his visit in Spring 2021 to the Institute of Geometry and Physics, USTC, where he, along with the other two authors, finalized this manuscript. B.X. extends deep appreciation to Qing Chen, Weibo Fu, Yi Ouyang, Mao Sheng and Michael Zieve for valuable conversations throughout the extended gestation period of this work. J.S. is partially supported by the National Natural Science Foundation of China (Grant Nos. 12001399 and 11831013) and the International Postdoctoral Exchange Fellowship Program by the Office of China Postdoctoral Council (NO. PC2021053). B.X. is supported in part by the National Natural Science Foundation of China (Grant Nos. 12271495, 11971450, and 12071449) and the CAS Project for Young Scientists in Basic Research (YSBR-001). Y.Y. is supported in part by the National Natural Science Foundation of China (Grant Nos. 11971449, 12131015, and 12161141001).

\end{document}